\documentclass[preprint,12pt]{elsarticle}

\usepackage[english]{babel}
\usepackage[utf8x]{inputenc}
\usepackage[T1]{fontenc}
\usepackage{amssymb,amsfonts,amsmath,latexsym}
\usepackage{hyperref}     % Para referenciar
\usepackage{graphicx}     % Para incluir graficos
\usepackage{xcolor}      % Para definir colores
\usepackage{setspace}    % Para incluir interlineados
\usepackage{booktabs}    % Para hacer tablas
\usepackage{longtable}   % Para el tipo de tablas utilizados
\usepackage{textcomp}    % Para algunos signos y simbolos
\usepackage{bm}         % Para escribir simbología en negrita

\usepackage{verbatim}

\def \ds{\displaystyle}

\newcommand{\R}{\mathbb R}
\newcommand{\C}{\mathbb C}

\newtheorem{thrm}{Theorem}
\newtheorem{lmm}{Lemma}
\newdefinition{rmrk}{Remark}
\newdefinition{nt}{Note}
\newproof{proof}{Proof}
\newtheorem{dfntn}{Definition}
\newtheorem{prpstn}{Proposition}
\newtheorem{xmpl}{Example}

\makeatletter
\def\ps@pprintTitle{}
\makeatother

\begin{document}

\begin{frontmatter}

%% Title, authors and addresses

%% use the tnoteref command within \title for footnotes;
%% use the tnotetext command for theassociated footnote;
%% use the fnref command within \author or \address for footnotes;
%% use the fntext command for theassociated footnote;
%% use the corref command within \author for corresponding author footnotes;
%% use the cortext command for theassociated footnote;
%% use the ead command for the email address,
%% and the form \ead[url] for the home page:
%% \title{Title\tnoteref{label1}}
%% \tnotetext[label1]{}
%% \author{Name\corref{cor1}\fnref{label2}}
%% \ead{email address}
%% \ead[url]{home page}
%% \fntext[label2]{}
%% \cortext[cor1]{}
%% \affiliation{organization={},
%%             addressline={},
%%             city={},
%%             postcode={},
%%             state={},
%%             country={}}
%% \fntext[label3]{}

\title{Regularization operators for identifying the unknown source in the time-fractional convection-diffusion-reaction equation}

%% use optional labels to link authors explicitly to addresses:
%% \author[label1,label2]{}
%% \affiliation[label1]{organization={},
%%             addressline={},
%%             city={},
%%             postcode={},
%%             state={},
%%             country={}}
%%
%% \affiliation[label2]{organization={},
%%             addressline={},
%%             city={},
%%             postcode={},
%%             state={},
%%             country={}}

%\author{Guillermo Federico Umbricht$^{a,b,*}$}
%\ead{guilleungs@yahoo.com.ar}

%\affiliation[1]{organization={Departamento de Matem\'atica, Facultad de Ciencias Empresariales, Universidad Austral},%Department and Organization
            %addressline={Paraguay 1950}, 
            %city={Rosario},
            %postcode={S2000FZF}, 
            %state={Santa Fe},
            %country={Argentina}}
						
%\affiliation[2]{organization={Consejo Nacional de Investigaciones Cient{\'i}ficas y T\'ecnicas (CONICET)},%Department and Organization
 %          addressline={Godoy Cruz 2290}, 
  %        city={CABA},
   %       postcode={C1425FQB}, 
    %      state={Buenos Aires},
     %     country={Argentina}}

\author{Guillermo Federico Umbricht$^{a,*}$}
\ead{GUmbricht@austral.edu.ar}

\affiliation{organization={Laboratorio de Investigaci\'on, Desarrollo y Transferencia de la Universidad Austral (LIDTUA), Facultad de Ingenier\'ia, Universidad Austral},%Department and Organization
            addressline={Mariano Acosta 1611}, 
            city={Pilar},
            postcode={B1629WWA}, 
            state={Buenos Aires},
            country={Argentina}}

\cortext[cor1]{Corresponding author}

\author{Diana Rubio$^{b}$}

\affiliation{organization={Instituto de Tecnolog\'ias Emergentes y Ciencias
Aplicadas (UNSAM-CONICET), Centro de Matem\'atica
Aplicada, Escuela de Ciencia y Tecnolog\'ia, Universidad
Nacional de General San Mart\'in},%Department and Organization
           addressline={25 de mayo y Francia}, 
          city={San Mart\'in},
          postcode={B1650}, 
          state={Buenos Aires},
          country={Argentina}}

\begin{abstract}

This article presents a mathematical study of the problem of identifying a time-dependent source term in transport processes described by a time-fractional parabolic equation, based on noisy time-dependent measurements taken at an arbitrary position. The problem is analytically solved using Fourier techniques, and it is shown that the solution is unstable. To address this instability, three one-parameter families of regularization operators are proposed, each designed to counteract the factors responsible for the instability of the inverse operator. Additionally, a new rule for selecting the regularization parameter is introduced, and an error bound is derived for each estimate. Numerical examples with varying characteristics are provided to illustrate the advantages of the proposed strategies.

\end{abstract}

%%Graphical abstract
%\begin{graphicalabstract}
%\includegraphics{grabs}
%\end{graphicalabstract}

%%Research highlights
%\begin{highlights}
%\item Research highlight 1
%\item Research highlight 2
%\end{highlights}

\begin{keyword}
%% keywords here, in the form: keyword \sep keyword
Time-fractional parabolic equation \sep Inverse problems \sep Ill-posed problems \sep Regularization operators.
%% PACS codes here, in the form: \PACS code \sep code

%% MSC codes here, in the form: \MSC code \sep code
%% or \MSC[2008] code \sep code (2000 is the default)
%\MSC 58J35 \sep 80A23 \sep 35R25 \sep 47A52
\end{keyword}

\end{frontmatter}

%% \linenumbers

%% main text
\section{Introduction}

In recent years, fractional differential equations have become a prominent area of research, attracting considerable attention from scholars. A comprehensive historical overview of the development of fractional differential operators can be found in \cite{Miller93, Oldham74}. This growing interest can be attributed to two key factors. Firstly, fractional derivatives extend the concept of integer-order derivatives, offering the potential to address challenges that integer-order derivatives cannot overcome \cite{Su10}. Secondly, the wide range of applications across various fields, including medicine, chemistry, biology, physics, electrical engineering, and mechanical engineering, as well as in fractional dynamics, control theory, signal processing, and system identification, have further fueled this interest. For examples, see \cite{Podlubny99, Metzler00, Scalas00}.

Time-fractional parabolic differential equations, especially the time-fractional diffusion equation, are derived from continuous-time random walks and are employed to model specific transport processes, often referred to in the literature as anomalous diffusion problems, including superdiffusion and subdiffusion \cite{Su10, Qiu20}. Recently, these equations have been explored both analytically and numerically in various contexts by several authors. Concrete examples can be found in \cite{Jiang11, Lin07, Luchko10, Luchko11, Murio08, Gul20}.

On the other hand, the problem of source determination has been extensively analyzed in recent years across various scientific and engineering fields, attracting considerable attention in current research. This problem has applications in areas such as heat conduction \cite{Alifa75, Elden00, Umbricht19d, Zhao11}, crack identification \cite{Zeng96}, electromagnetic theory \cite{Zeng96, Banks15, Ciaret10, Elbadia00}, geophysical prospecting \cite{Beroza88}, contaminant detection \cite{Li06}, and tumor cell detection \cite{Macleod99}, among others. The literature offers several methods for source determination, with important tools including the logarithmic potential method \cite{Ohe94}, the projective method \cite{Nara03}, Green's functions \cite{Hon10}, dual reciprocity boundary element methods \cite{Farcas03,Sun97} and the method of fundamental solutions \cite{Jin07}, among others.

The problem of source determination is considered ill-posed in the sense of Hadamard \cite{Hada24}, as the solution does not depend continuously on the data. Regularization methods \cite{Engl96, Kirsch11} are essential for estimating unstable solutions. Among the most commonly used methods are the iterative regularization method \cite{Honchbruck09, Johansson08}, the simplified Tikhonov regularization method \cite{Fu04, Cheng07, Cheng08}, the modified regularization method \cite{Natterer84, Yang10b, Yang10, Yang11, Zhao14}, Fourier truncation \cite{Yang11}, and the mollification method \cite{Yang14b}.

Regarding the use of regularization methods for estimating the source in classical parabolic equations, numerous studies address the problem specifically for the diffusion equation, utilizing various approaches, characteristics, and data \cite{Elden00, Zhao14, Yang14b, Yan10b, Dou09, Dou09b, Martin96, Trong05, Trong06, Farcas06, Ahmadabadi09, Johansson07, Liu09}. Fewer articles focus on the estimation of the source term in complete classical parabolic equations \cite{Umbricht19a, Umbricht19b, Umbricht21, Umbricht24, Li19, Le20, Lukyanenko23}.  

Concerning the identification of an unknown source in a fractional parabolic equation using a regularization method, only a few works address the problem for a fractional diffusion equation. In \cite{Yang14, Murio08b}, the authors use the mollification method to determine the time-dependent source in the fractional diffusion equation. In \cite{Yang15}, they estimate the time-dependent source using the quasi-reversibility regularization method. In \cite{Zhang11}, the authors employ analytical continuation and the Laplace transform to demonstrate the uniqueness of identifying the unknown source, which depends solely on the spatial variable, in the fractional diffusion equation within a bounded domain. In \cite{Wei10}, the coupled method is used to identify the unknown source for spatial fractional diffusion equations. In \cite{Chi11}, an optimal perturbation regularization algorithm is employed to identify the unknown source, which depends only on the spatial variable, in spatial fractional diffusion equations. In \cite{Wei13}, the boundary element method combined with a generalized Tikhonov regularization method is used to identify the unknown source, which depends only on the temporal variable, in the time-fractional diffusion equation. To the best of the authors' knowledge, there are no articles that estimate the source using any regularization method in a time-fractional convection-diffusion-reaction equation. Addressing this problem is crucial since fractional diffusion equations alone are insufficient to model all processes occurring in nature.

This paper addresses the estimation of the unknown source term in a time-fractional convection-diffusion-reaction equation using time-dependent input data. This problem is ill-posed because high-frequency components of arbitrarily small errors in the data can lead to disproportionately large errors in the solution. To mitigate this challenge, three families of one-parameter regularization operators are introduced to counteract the instability of the inverse operator. These regularization operators reformulate the ill-posed problem into families of well-posed problems that approximate the original one. Additionally, a new rule for the a priori selection of the regularization parameter is proposed, which is independent of the norm of the function to be estimated. The stability and convergence of each method are analyzed, and an optimal bound for the estimation error is derived.

The three strategies presented here were recently employed for source term estimation in a classical linear parabolic equation \cite{Umbricht24}. In this work, we generalize those methods to a time-fractional convection-diffusion-reaction equation, where the input measurements or data are time-dependent and can be taken at any position. To demonstrate the effectiveness of the proposed regularization methods and compare the performance of the introduced regularization operators, numerical examples with different characteristics are provided.

\section{Source identification}

In this section, we formally present the inverse problem under study, provide an analytical expression for its solution, and demonstrate that the problem is ill-posed.

\subsection{Presentation of the problem} 

Let $u:\R^+ \times \R^+ \rightarrow \R$, we aim to determine the source term in the following linear time-fractional parabolic equation on an unbounded domain.
\begin{equation}
\partial_{0^+}^{\alpha} u(x,t)=\omega \,u_{xx}(x,t)-\beta \, u_{x}(x,t)-\nu \, u(x,t)+f(t), \qquad  x>0, \,\,\, t>0,
\label{Ec_1_Mod}
\end{equation}
where the time-fractional derivative $\partial_{0^+}^{\alpha}$ is the Caputo fractional derivative of order $\alpha$ $(0<\alpha\leq1)$ defined by \cite{Caputo67}
\begin{equation}
\partial_{0^+}^{\alpha}u(x,t)=
\begin{cases}
\dfrac{1}{\Gamma(1-\alpha)} \ds \int_0^t \dfrac{\partial u(x,s)}{\partial s} \dfrac{ds}{(t-s)^\alpha}, & \quad 0<\alpha<1, \vspace{0.25cm} \\
\dfrac{\partial u(x,t)}{\partial t},  & \quad \alpha=1,  
\label{Ec.Fracc}
\end{cases}
\end{equation}
with $\Gamma$ being the Gamma function defined by:
\begin{equation}
\Gamma(z)=\ds \int_{0}^{\infty} t^{z-1} e^{-t} \, dt.
\label{Gamma}
\end{equation}

Without loss of generality, we assume that the initial and boundary conditions are zero, i.e.,
\begin{equation}
\label{C_I}
\begin{cases}
u(x,t)=0,  & \quad x \geq 0, \,\,\, t=0, \\
u(x,t)=0, & \quad x =0, \,\,\, t\geq0,
\end{cases}
\end{equation}
additionally, we assumed that
\begin{equation}
\label{Cota_u}
u(x,t) \textit{ is bounded in } [0,+\infty) \times [0,+\infty).
\end{equation}
The determination of the source term $f(t)$ in \eqref{Ec_1_Mod}-\eqref{Gamma}, under the conditions \eqref{C_I}-\eqref{Cota_u}, is performed using experimental or simulated noisy data, at a specific position $x_0$. 
\begin{equation}
\label{Med_Ruidosas}
u(x_0,t)=y(t), \qquad y_{\delta}(t)=y(t) + \textit{noise}, \,\,\, t\geq0,  
\end{equation}
where $y_{\delta}$ represents the noisy data or measurements, and $\delta$ denotes the noise in the data. It is also assumed that this noise is bounded, that is,
\begin{equation} 
\label{noiselevel}
    ||y(t)-y_{\delta}(t)||_{L^2(\R)}\le \delta, \qquad 0<\delta < \delta_M,\\
\end{equation} 
wich $\delta_M \in \R^{+} $ is the maximum noise level tolerated. In practice, $\delta_M$ is determined based on  measurement, instrumentation and calibration errors.

\subsection{Problem solution}

The source estimation problem is solved using the Fourier transform, which is included here for the sake of completeness.

\begin{dfntn}\label{DefTransf} 

 Let $g \in L^2(\R)$, the Fourier transform \cite{Marks09} is defined by
\begin{equation*}
\widehat{g}(\xi) :=\dfrac{1}{\sqrt{2\pi}} \ds\int\limits_{-\infty}^{\infty}{e^{-i \xi t} \, g (t) \,  dt}, \qquad   \xi \in \R.
\end{equation*}

Let $\widehat{g} \in L^2(\R)$, the Fourier antitransform is defined by
\begin{equation*}
\label{defantitransfn}
g(t) :=\dfrac{1}{\sqrt{2\pi}}\ds\int\limits_{-\infty}^{\infty}{e^{i \xi t} \, \widehat{g} (\xi) \,  d\xi}, \qquad  t \in \R.
\end{equation*}
\end{dfntn}

\begin{nt}
%\textcolor{red}{
%The Fourier transform in $L^{2}$ is the extension of the Fourier transform to square-integrable functions. It is defined as an operator that maps a function $f \in L^{2}$ to its transform $\widehat{f} \in L^{2}$, preserving the norm of the function and endowing the operator with the properties of being linear, bounded, bijective, and unitary. This operator is constructed by considering a dense subset of $L^{2}$ (such as $L^{1} \cap L^{2}$), where the Fourier transform is well defined, and then extending it by continuity to the whole space $L^{2}$.
%}

The Fourier transform on $L^{2}$ is defined as the unitary extension of the Fourier transform from $L^{1} \cap L^{2}$ to the whole space $L^{2}$. It preserves the $L^{2}$-norm and is a linear, bounded, bijective, and unitary operator.
\end{nt}

By using Definition \ref{DefTransf}, we can find the analytical solution to the problem of interest given by \eqref{Ec_1_Mod}-\eqref{noiselevel}. The solution is presented in the following result.

\begin{thrm}[Source identification]\label{teorema_Fuente_Inversa}
Let $0<\alpha\leq1$ and the parameters $\omega,\beta, \nu ,x_0, \delta, \delta_M \in \R^+$ with $\delta<\delta_M $. Consider functions $u(x, \cdot), f(\cdot), y(\cdot), y_\delta(\cdot) \in L^2(\R)$ such that $||y-y_{\delta}||_{L^2(\R)}\le \delta$ and $y_{\delta}(t)=y(t) + \textit{noise}$ satisfying the parabolic problem
\begin{equation}
\label{transpeqn}
\begin{cases} 
\partial_{0^+}^{\alpha} u(x,t)=\omega \,u_{xx}(x,t)-\beta \, u_{x}(x,t)-\nu \, u(x,t)+f(t), & \quad  x>0, \,\,\, t>0, \\
u(x,t)=0,  & \quad x \geq 0, \,\,\, t=0, \\
u(x,t)=0, & \quad x =0, \,\,\, t\geq0, \\
u(x,t)\big|_{x\rightarrow \infty} \text {bounded},  & \quad t \geq 0\\
u(x,t)=y(t), & \quad  x=x_0, \,\,\, t\geq0,
\end{cases}
\end{equation}

Then, the expression for the source in the Fourier variable is given by

\begin{equation*}
\widehat{f}(\xi )=\Lambda (\xi )\widehat{y}(\xi ),
\end{equation*}
where
\begin{equation*}
\Lambda (\xi )=\dfrac{z(\xi)}{1-e^{-h(\xi)\, x_0}}, \qquad h(\xi) = \dfrac{-\beta+\sqrt{\beta^2+4\,\omega \, z(\xi) }}{2\,\omega}, \qquad z(\xi)= \nu+(i \, \xi)^{\alpha}.
\end{equation*}

\end{thrm}

\begin{proof}
In the proof of Theorem \ref{teorema_Fuente_Inversa}, the Fourier transform defined in \ref{DefTransf} is applied to the time variable of the system described by \eqref{transpeqn}, using the fact that $\partial_{0^+}^{\alpha} u(x,t) = (i \, \xi)^{\alpha} \, \widehat{u}(x,\xi)$ \cite{Podlubny99}.

\begin{equation}
\label{transpeqn2}
\begin{cases} 
(i \, \xi)^{\alpha} \, \widehat{u}(x,\xi)=\omega \,\widehat{u}_{xx}(x,\xi)-\beta \, \widehat{u}_{x}(x,\xi)-\nu \, \widehat{u}(x,\xi)+\widehat{f}(\xi), & \,  x>0, \, \xi \in \R, \\
\widehat{u}(x,\xi)=0, & \, x =0, \, \xi \in \R, \\
\widehat{u}(x,\xi)\big|_{x\rightarrow \infty} \textit{bounded},  & \xi \in \R, \\
\widehat{u}(x,\xi)=\widehat{y}(\xi), & \,  x=x_0, \, \xi \in \R, 
\end{cases}
\end{equation}
where $\xi$ is the Fourier variable. Equivalently, the identification problem given by \eqref{transpeqn} can be reformulated in the frequency domain using \eqref{transpeqn2}, as follows:
\begin{equation}
\label{transpeqn3}
\begin{cases} 
\omega \,\widehat{u}_{xx}(x,\xi)-\beta \, \widehat{u}_{x}(x,\xi)-z(\xi) \, \widehat{u}(x,\xi)=-\widehat{f}(\xi), & \,  x>0, \, \xi \in \R, \\
\widehat{u}(x,\xi)=0, & \, x =0, \, \xi \in \R, \\
\widehat{u}(x,\xi)\big|_{x\rightarrow \infty} \textit{bounded},  & \xi \in \R, \\
\widehat{u}(x,\xi)=\widehat{y}(\xi), & \,  x=x_0, \, \xi \in \R, 
\end{cases}
\end{equation}
where $z(\xi)$ is defined by the following expression:

\begin{equation}
\label{z}
z(\xi)= \nu+(i \, \xi)^{\alpha}.
\end{equation}

The system described by \eqref{transpeqn3} is represented by a second-order non-homogeneous differential equation with an initial condition. The analytical solution to this equation is given by,
\begin{equation}
\label{solutionu}
\widehat{u}(x,\xi) = \dfrac{1-e^{-h(\xi)\, x}}{z(\xi)}\widehat{f}(\xi),
\end{equation}
where
\begin{equation}
\label{h}
h(\xi) = \dfrac{-\beta+\sqrt{\beta^2+4\,\omega \, z(\xi) }}{2\,\omega}.
\end{equation}

Since $\widehat{u}(x_0,\xi)=\widehat{y}(\xi)$, evaluating \eqref{solutionu} at $x=x_0$ produces an expression for the source in the frequency domain, 

\begin{equation}
\label{illf}
\widehat{f}(\xi )=\Lambda (\xi )\widehat{y}(\xi ),
\end{equation}
where $\Lambda$ denotes the inverse operator defined by:
\begin{equation}
\label{Lambda}
\Lambda (\xi )=\dfrac{z(\xi)}{1-e^{-h(\xi)\, x_0}}.
\end{equation}
%
%Equivalently, applying the Definition \ref{DefTransf} (of the Fourier antitransform) \textcolor{blue}{to the results \eqref{illf} and, as the available data contain noise, it follows that:} 
%\begin{equation}
%\label{solucion_fuente0}
%f(t) =\dfrac{1}{\sqrt{2\pi}} \ds\int\limits_{-\infty}^{\infty} e^{i \xi  t} \Lambda (\xi) \left[\dfrac{1}{\sqrt{2\pi}} \ds\int\limits_{-\infty}^{\infty} e^{-i \xi  t} y(t)dt\right]d\xi.
%\end{equation}
%
%
%\begin{equation}
%\label{solucion_fuente}
%f_{\delta}(t) =\dfrac{1}{\sqrt{2\pi}} \ds\int\limits_{-\infty}^{\infty} e^{i \xi  t} \Lambda (\xi) \left[\dfrac{1}{\sqrt{2\pi}} \ds\int\limits_{-\infty}^{\infty} e^{-i \xi  t} y_{\delta}(t)dt\right]d\xi,
%\end{equation}
%
which concludes the proof.

\end{proof}

\begin{rmrk}
For the particular case $x_0=1$, $\omega=1$ and $\beta=\nu=0$, it follows that $z(\xi)=(i\,\xi)^\alpha$, $h(\xi)=(i\,\xi)^{\alpha/2}$ and therefore $\Lambda(\xi)=\dfrac{(i\,\xi)^\alpha}{1-e^{-(i\,\xi)^{\alpha/2}}}$. This result coincides with the findings in \cite{Yang14,Yang15}, where the estimation of the unknown source in the time-fractional diffusion equation is studied.
\end{rmrk}

\subsection{Ill-posed problem}

We now demonstrate that the problem of identifying the time-dependent source in a time-fractional parabolic equation from noisy measurements is ill-posed in the sense of Hadamard \cite{Hada24}, as the solution does not depend continuously on the data. This ill-posedness arises because the inverse operator $\Lambda$ (defined in \eqref{Lambda}) is unbounded, as will be shown below.

\begin{thrm} [The problem is ill-posed]\label{Teo_Prob_mal_Planteado}

Under the assumptions given in \eqref{Ec_1_Mod}-\eqref{noiselevel}, the identification problem \eqref{transpeqn} presented in Theorem \ref{teorema_Fuente_Inversa} is an ill-posed problem in the sense of Hadamard, as the solutions do not vary continuously with respect to the data.
\end{thrm}

\begin{proof}

We denote $\widehat{f_{\delta}}(\xi)=\Lambda (\xi )\widehat{y_{\delta}}(\xi )$. It is evident that
\begin{equation}
\label{error_f_fdelta}
\|\widehat{f}(\xi) - \widehat{f_{\delta}}(\xi)\|_{L^2(\R)} = \| \Lambda(\xi)  \widehat{y}(\xi )- \Lambda(\xi) \widehat{y_{\delta}}(\xi )\|_{L^2(\R)}=
\left\| \Lambda(\xi)(\widehat{y}(\xi )- \widehat{y_{\delta}}(\xi ))\right\|_{L^2(\R)},
\end{equation}
on the other hand, by using \eqref{z}, \eqref{h}, and \eqref{Lambda}, we find that,

\begin{equation}
\begin{split}
 \label{lim}
| \Lambda (\xi )| &= \left| \dfrac{  z(\xi)}{1-e^{-h(\xi)\, x_0}}\right|
=\left| \dfrac{(i \xi)^{\alpha}+\nu}{1-e^{-\left(\frac{-\beta+\sqrt{\beta^2+4\omega[(i \xi)^{\alpha}+\nu]}}{2\omega}\right)\, x_0}}\right| 
= \dfrac{\left|(i \xi)^{\alpha}+\nu\right|}{\left|1-e^{-\left(\frac{-\beta+\sqrt{\beta^2+4\omega[(i \xi)^{\alpha}+\nu]}}{2\omega}\right)\, x_0}\right|}\\
& \geq \dfrac{\left|(i \xi)^{\alpha}+\nu\right|}{1+e^{-\Re\left(\frac{-\beta+\sqrt{\beta^2+4\omega[(i \xi)^{\alpha}+\nu]}}{2\omega}\, x_0\right)}}=\dfrac{\left|(i \xi)^{\alpha}+\nu\right|}{1+e^{-\Re\left(\frac{-\beta+\sqrt{\beta^2+4\omega \nu+ 4\omega (i \xi)^{\alpha}}}{2\omega}\, x_0\right)}},
\end{split}
\end{equation}
where
\begin{equation}
\label{i_xi_alpha} (i \xi)^{\alpha}=
\begin{cases}
\left|\xi\right|^{\alpha}\left(\cos(\frac{\alpha \,\pi}{2})+i \sin(\frac{\alpha \,\pi}{2}) \right), \qquad \xi\geq0, \\
\left|\xi\right|^{\alpha}\left(\cos(\frac{\alpha \,\pi}{2})-i \sin(\frac{\alpha \,\pi}{2}) \right), \qquad \xi<0.
\end{cases}
\end{equation}

By substituting \eqref{i_xi_alpha} into \eqref{lim}, we find that $\Lambda (\xi)$ is unbounded when  $|\xi|$ tends to infinity. As shown in \eqref{error_f_fdelta}, this unboundedness amplifies the error in measurements at high frequencies, potentially resulting in a significant estimation error $ \|\widehat{f} - \widehat{f_{\delta}}\|_{L^2(\R)}$, even for very small observation or measurement errors. In other words, the solution to the identification problem \eqref{transpeqn} does not vary continuously with respect to the data (see \cite{Engl96}).
\end{proof}

Not only can it be shown that the operator $\Lambda$ is unbounded, but its asymptotic behavior at infinity can also be characterized, as established in the following result.

\begin{prpstn}

The inverse operator $\Lambda$ defined in \eqref{Lambda} satisfies 

\begin{equation}
\dfrac{\left|(i \xi)^{\alpha}+\nu\right|}{1+e^{-\frac{x_0}{2\omega}\Re\left(-\beta+\sqrt{\beta^2+4\omega \nu+ 4\omega (i \xi)^{\alpha}}\right)}}\le \left\vert \Lambda(\xi)\right\vert 
 \leq\dfrac{\nu+|\xi|^{\alpha}}{ 1-e^{-\frac{x_0}{2\omega} \left(-\beta+ \sqrt{\beta^2+4\omega \nu}\right)}}.
\end{equation}

\begin{proof}
The lower bound was established in Equation \eqref{lim} of Theorem \ref{Teo_Prob_mal_Planteado}. In what follows, the corresponding upper bound is derived. Lemma \ref{lemma1}, applied to the absolute value of \eqref{Lambda} is used to obtain,
\begin{equation*}
\begin{split}
& \left\vert \Lambda(\xi)\right\vert 
= \dfrac{\left|\nu + (i\xi)^\alpha\right|}{\left\vert 1-e^{-\frac{x_0}{2\omega} \left(-\beta+\sqrt{\beta^2+4\omega[\nu+(i\xi)^\alpha]}\right)}\right\vert}
\\ &\le 
\dfrac{\nu+|\xi|^{\alpha}}{\left\vert 1-e^{-\frac{x_0}{2\omega} \left(-\beta+\sqrt{\beta^2+4\omega[\nu+\left|\xi\right|^{\alpha}\left(\cos\left(\frac{\alpha \,\pi}{2}\right) \pm i \sin\left(\frac{\alpha \,\pi}{2}\right) \right)]}\right)}\right\vert}
\\ & = 
\dfrac{\nu+|\xi|^{\alpha}}{\left\vert 1-e^{-\frac{x_0}{2\omega} \left(-\beta+\sqrt{\beta^2+4\omega\nu+4\omega\left|\xi\right|^{\alpha}\cos\left(\frac{\alpha \,\pi}{2}\right) \pm 4\omega\left|\xi\right|^{\alpha} i \sin\left(\frac{\alpha \,\pi}{2}\right) }\right)}\right\vert}
\\ & \le 
\dfrac{\nu+|\xi|^{\alpha}}{ 1-e^{-\frac{x_0}{2\omega} \left(-\beta+ \Re \left(\sqrt{\beta^2+4\omega\nu+4\omega\left|\xi\right|^{\alpha}\cos\left(\frac{\alpha \,\pi}{2}\right) \pm 4\omega\left|\xi\right|^{\alpha} i \sin\left(\frac{\alpha \,\pi}{2}\right) }\right)\right)}}
\\ & \le  \dfrac{\nu+|\xi|^{\alpha}}{ 1-e^{-\frac{x_0}{2\omega} \left(-\beta+ \sqrt{\Re \left(\beta^2+4\omega\nu+4\omega\left|\xi\right|^{\alpha}\cos\left(\frac{\alpha \,\pi}{2}\right) \pm 4\omega\left|\xi\right|^{\alpha} i \sin\left(\frac{\alpha \,\pi}{2}\right)\right)}\right)}}
\\ & =  \dfrac{\nu+|\xi|^{\alpha}}{ 1-e^{-\frac{x_0}{2\omega} \left(-\beta+ \sqrt{\beta^2+4\omega\nu+4\omega\left|\xi\right|^{\alpha}\cos\left(\frac{\alpha \,\pi}{2}\right)}\right)}}
\\ & = \dfrac{\nu+|\xi|^{\alpha}}{1-e^{-\frac{x_0}{2\omega} \left(-\beta+ \sqrt{\beta^2+4\omega \left(\nu+\left|\xi\right|^{\alpha}\cos\left(\frac{\alpha \,\pi}{2}\right)\right)}\right)}}.
\end{split}
\end{equation*}

Therefore, it follows that
\begin{equation}
\label{D_ineq}
 \left\vert \Lambda(\xi)\right\vert \le \dfrac{\nu+|\xi|^{\alpha}}{ 1-e^{-\frac{x_0}{2\omega} \left(-\beta+ \sqrt{\beta^2+4\omega \nu}\right)}}
,
\end{equation}

which completes the proof.

\end{proof}
\end{prpstn}

\section{Regularization operators}

When an inverse problem is ill-posed, a regularization method is typically employed to stabilize the solution. In this section, we propose three regularization operators for comparative purposes. We demonstrate the existence of regularization parameters that lead to three convergent methods and address basic theoretical issues related to regularization operators. Readers unfamiliar with this topic may find additional information in \cite{Engl96, Kirsch11}.

\subsection{Regularization solutions}

To stabilize the ill-posed problem, we utilize regularization operators.

\begin{dfntn}
Let $\mathbb{X}$ and $\mathbb{Y}$ be Hilbert spaces, and let $T : \mathbb{Y} \longrightarrow \mathbb{X}$ a linear unbounded operator. A regularization strategy for $T$ is defined as a family of linear bounded operators that satisfies
\begin{equation*}
\label{DefinitionR}
R_{\mu} : \mathbb{Y} \longrightarrow \mathbb{X}, \quad \mu>0, \quad / \lim_{\mu \to 0^+} R_{\mu} y = Ty, \quad \forall y \in \mathbb{Y}.
\end{equation*}
\end{dfntn}

Since the inverse operator \eqref{Lambda} has a functional form similar to that obtained for the parabolic problem studied in \cite{Umbricht24}, we propose using the same three uniparametric families of linear operators, i.e.,
$R_\mu ^i: L^2(\R) \to L^2(\R)$ with $\mu \in \R^{+}$ and $i=1,2,3$; such that
\begin{equation}
\label{familyR}
(R_{\mu}^{1} \, \widehat{y})(\xi) := \dfrac{\Lambda (\xi)}{1+\mu^2 \, \xi^2} \, \widehat{y}(\xi) , \qquad
(R_{\mu}^{2} \, \widehat{y})(\xi) := \dfrac{\Lambda (\xi)}{1+\mu^2 \, \xi^4} \, \widehat{y}(\xi) , \qquad
(R_{\mu}^{3} \, \widehat{y})(\xi) := \dfrac{\Lambda (\xi)}{e^{\mu^2 \,\xi^2/4}} \, \widehat{y}(\xi) ,
\end{equation}
where $\Lambda (\xi)$ is defined in \eqref{Lambda}, $R_{\mu}^{i}$ with $i=1,2,3$ represents regularization strategies for $\Lambda (\xi)$, and $\mu$ is the regularization parameter.

\begin{nt}
Note that the denominators in the expressions \eqref{familyR}, were introduced to ensure that the resulting linear operators $R_{\mu}^{i}$ (for $i=1,2,3$) stabilize the solution for the source.
\end{nt}

\begin{thrm}  [Convergent regularization operators] 
Consider the source identification problem \eqref{transpeqn}.
Let $u(x,\cdot), f(\cdot) \in L^2(\R) $ satisfy the following system
\begin{equation}
\label{illpp}
\begin{cases}
\partial_{0^+}^{\alpha} u(x,t)=\omega \,u_{xx}(x,t)-\beta \, u_{x}(x,t)-\nu \, u(x,t)+f(t), & \quad  x>0, \,\,\, t>0, \\
u(x,t)=0,  & \quad x \geq 0, \,\,\, t=0, \\
u(x,t)=0, & \quad x =0, \,\,\, t\geq0, \\
u(x,t)\big|_{x\rightarrow \infty} \text {bounded},  & \quad t \geq 0.
\end{cases}
\end{equation}
and let $\{R_{\mu}^{i}\}$ with $i=1,2,3$ denote the families of operators defined in \eqref{familyR}.
Then, for all $y(t)=u(x_0,t)$ there exists an a priori parameter choice rule for $\mu>0$ such that the pairs $(R_{\mu}^ {i},\mu)$ for $i=1,2,3$ constitute convergent regularization methods for the identification problem \eqref{illpp}.
\end{thrm}
%%%%%
%
\begin{proof}
The operators $\{R_{\mu}^{i}\}$ with $i=1,2,3$ defined in \eqref{familyR} are continuous for all $\xi \in \R$ and are bounded, that is,
\begin{equation*}
\begin{split}
0 \leq \lim_{|\xi| \to \infty} \left|R_{\mu}^{i}\right|<\infty, \qquad i=1,2,3.
\end{split}
\end{equation*}

Therefore, for each parameter $\mu>0$, the operators $R_{\mu}^{i}$ with $i=1,2,3$ are linear, continuous and satisfy the following pointwise convergence:
\begin{equation*}
\lim _{\mu \to 0^+} (R_{\mu}^{i} \, \widehat{y})(\xi) = (\Lambda \, \widehat{y}) (\xi), \qquad  i=1,2,3,
\end{equation*}
where $ \widehat{y} \in L^{2}(\R)$, then $R_{\mu}^{i}$ with $i=1,2,3$ are regularization strategies for $ \Lambda$. Therefore, according to Proposition~3.4 in~\cite{Engl96}, there exist \emph{a priori} parameter choice rules $\mu$ such that $(R_{\mu}^{i}, \mu)$ with $i=1,2,3$ constitute convergent regularization methods (see Definition~3.1 in~\cite{Engl96}) for solving \eqref {illf}. Theorem \ref{boundestimate} will provide a rule of choice for the regularization parameter.
\end{proof}

The regularized solution of the inverse source identification problem in frequency space, is given by
\begin{equation}
\label{ffregtransf}
\widehat{f}_{\delta ,\mu}^{i} (\xi) = (R_{\mu}^{i}  \, \widehat{y}_\delta) (\xi), \qquad i=1,2,3.
\end{equation}

\subsection{Error Analysis}

In this section, we will derive a bound for the error associated with each estimation.

\subsubsection{Auxiliary results}

To analyze the behavior of the proposed regularizations, we first introduce some results that will be used later to obtain an error bound between the source $f(t)$
 and each estimate $f_{\delta ,\mu } ^{i}(t).$

%\begin{rmrk} 
%Some of the auxiliary results may seem trivial to the reader. However, to ensure a complete and self-contained work, a brief demonstration is provided for each of them.
%\end{rmrk}

\begin{lmm}
\label{lemma3}
Let $\rho \in \R^+$ and $0<\alpha\leq1$ . If $0<\mu<1 $ the following inequalities hold
\end{lmm}
\begin{equation*}
\dfrac{\rho^\alpha}{1+\rho^2\mu^2} < \dfrac{n_1}{\mu^2},  \qquad  \dfrac{\rho^\alpha}{1+\rho^4\mu^2} < \dfrac{n_2}{\mu^2}, 
\qquad  \dfrac{\rho^\alpha}{e^{\rho^2\mu^2/4}} < \dfrac{n_3}{\mu^2},
\end{equation*}
where $n_1=n_1(\alpha)$, $n_2=n_2(\alpha)$ and $n_3=n_3(\alpha)$, are
\begin{equation}
\label{n's}
n_1=\dfrac{2-\alpha}{2} \left(\dfrac{\alpha}{2-\alpha}\right)^{\alpha/2},  \,  n_2=\dfrac{4-\alpha}{4} \left(\dfrac{\alpha}{4-\alpha}\right)^{\alpha/4}, \, n_3=\dfrac{(2 \alpha)^{\alpha/2}}{e^{\alpha/2}}.
\end{equation}
\begin{proof}
Let us defined: $\theta_1,\theta_2,\theta_3:\R^+\rightarrow\R$, such that
\begin{equation*}
\theta_1(\rho)=\dfrac{\rho^\alpha}{1+\rho^2\mu^2},  \qquad  \theta_2(\rho)=\dfrac{\rho^\alpha}{1+\rho^4\mu^2},  \qquad
\theta_3(\rho)=\dfrac{\rho^\alpha}{e^{\rho^2\mu^2/4}}.
\end{equation*}
The derivatives of the functions $\theta_i$ $i=1,2,3$ are obtained analytically
\begin{equation*}
\begin{split}
&\theta_1'(\rho)=\dfrac{\rho^{\alpha-1}\left[\alpha+\rho^2\mu^2(\alpha-2)\right]}{(1+\rho^2\mu^2)^2},  \, \\ &\theta_2'(\rho)=\dfrac{\rho^{\alpha-1}\left[\alpha+\rho^4\mu^2(\alpha-4)\right]}{(1+\rho^4\mu^2)^2},  \, \\
&\theta_3'(\rho)=\rho^{\alpha-1}e^{-\rho^2\mu^2/4}\left(\alpha-\dfrac{\rho^2\mu^2}{2}\right).
\end{split}
\end{equation*}
Then, each of the functions $\theta_i$, for $i=1,2,3$, has a unique critical point given by:
\begin{equation*}
{\rho_1}^*=\dfrac{1}{\mu} \left(\dfrac{\alpha}{2-\alpha}\right)^{1/2}, \quad {\rho_2}^*=\dfrac{1}{\sqrt{\mu}} \left(\dfrac{\alpha}{4-\alpha}\right)^{1/4},  \quad {\rho_3}^*=\dfrac{1}{\mu} \sqrt{2\alpha}.
\end{equation*}
Since the functions $\theta_i(\rho)\geq0$ $\forall \rho \in \R^+$, $\theta_i(0^+)=0$, $\ds \lim_{\rho \to +\infty} \theta_i(\rho)=0$. It follows then that ${\rho_i}^*$ is the absolute maximum of $\theta_i(\rho)$, this implies that:
\begin{equation*}
\dfrac{\rho^\alpha}{1+\rho^2\mu^2}=\theta_1(\rho)\leq \theta_1({\rho_1}^*)=\dfrac{1}{\mu^\alpha} \dfrac{2-\alpha}{2} \left(\dfrac{\alpha}{2-\alpha}\right)^{\alpha/2}<\dfrac{1}{\mu^2} \dfrac{2-\alpha}{2} \left(\dfrac{\alpha}{2-\alpha}\right)^{\alpha/2}= \dfrac{n_1}{\mu^2}, 
\end{equation*}
\begin{equation*}
\dfrac{\rho^\alpha}{1+\rho^4\mu^2}=\theta_2(\rho)\leq \theta_2({\rho_2}^*)=\dfrac{1}{\mu^{\alpha/2}} \dfrac{4-\alpha}{4} \left(\dfrac{\alpha}{4-\alpha}\right)^{\alpha/4}<\dfrac{1}{\mu^2} \dfrac{4-\alpha}{4} \left(\dfrac{\alpha}{4-\alpha}\right)^{\alpha/4}= \dfrac{n_2}{\mu^2}, 
\end{equation*}
\begin{equation*}
\dfrac{\rho^\alpha}{e^{\rho^2\mu^2/4}}=\theta_3(\rho)\leq \theta_3({\rho_3}^*)=\dfrac{1}{\mu^{\alpha}} \left(\dfrac{2\alpha}{e}\right)^{\alpha/2} <\dfrac{1}{\mu^{2}} \left(\dfrac{2\alpha}{e}\right)^{\alpha/2}= \dfrac{n_3}{\mu^2},
\end{equation*}
which concludes the proof.
\end{proof}

\begin{lmm} 
\label{lemma4}
Let $\omega, \beta, \nu, x_0 >0$; $0<\mu <1$; $0<\alpha\leq1$  and $R_{\mu}^{i}$ with $i=1,2,3$ given by \eqref{familyR}. Hence
\begin{equation*}
\left\vert  R_{\mu}^{i}(\xi) \right\vert < \dfrac{M_i}{\mu^2}, \quad i=1,2,3,
\end{equation*}
where
\begin{equation}
\label{M's}
M_i= 2 \left(\nu+n_i\right) \max \left\{1,\dfrac{1}{N}\right\},  \quad i=1,2,3,
\end{equation}
with $n_i$ defined in \eqref{n's} for $ i=1,2,3,$ and $N$ is given by:
\begin{equation}
\label{N}
N=N(\omega,\beta,\nu,x_0)= \dfrac{x_0}{2\omega}\left(-\beta+\sqrt{\beta^2+4\omega\nu}\right)
\end{equation}
\end{lmm} 
%%%%%
%
\begin{proof}
The three regularization operators defined in \eqref{familyR} incorporate the inverse operator $\Lambda(\xi)$, that is bounded as specified in \eqref{D_ineq} or, equivalently, 
\begin{equation} 
\label{D_ineq2}
|\Lambda(\xi)|
\leq \dfrac{\nu+|\xi|^{\alpha}}{ 1-e^{-N}}. 
\end{equation}

\setlength{\leftskip}{0pt}
\setlength{\leftskip}{0pt}

%\vspace{-0.8cm}
% 
\emph{If $N \ge 1$}: From the inequality \eqref{D_ineq2}, using Lemma \ref{lemma3} and Lemma \ref{lemma2}, we have,
%
%\vspace{-0.8cm}

%\setlength{\leftskip}{1cm}
%
\begin{equation} \label{part1_1}
\begin{split}
\phantom{space} \left|R_{\mu}^{1}(\xi)\right| &= \left|\dfrac{\Lambda(\xi)}{1+\xi^2\mu^2}\right|
\leq \dfrac{\nu+|\xi|^{\alpha}}{ (1-e^{-N})(1+\xi^2\mu^2)} 
< 2 \,\dfrac{\nu+|\xi|^{\alpha}}{ 1+\xi^2\mu^2}\\ 
&=2 \left(\dfrac{\nu}{ 1+\xi^2\mu^2}+\dfrac{|\xi|^{\alpha}}{ 1+\xi^2\mu^2}\right)
< 2 \left(\dfrac{\nu}{\mu^2}+\dfrac{n_1}{ \mu^2}\right) = \frac{2}{\mu^2} \left(\nu+n_1\right)
\end{split}
\end{equation}
\begin{equation} \label{part1_2}
\begin{split}
\phantom{space} \left|R_{\mu}^{2}(\xi)\right| &= \left|\dfrac{\Lambda(\xi)}{1+\xi^4\mu^2}\right|
\leq \dfrac{\nu+|\xi|^{\alpha}}{ (1-e^{-N})(1+\xi^4\mu^2)} 
< 2 \,\dfrac{\nu+|\xi|^{\alpha}}{ 1+\xi^4\mu^2}\\ 
&=2 \left(\dfrac{\nu}{ 1+\xi^4\mu^2}+\dfrac{|\xi|^{\alpha}}{ 1+\xi^4\mu^2}\right)
< 2 \left(\dfrac{\nu}{\mu^2}+\dfrac{n_2}{ \mu^2}\right) = \frac{2}{\mu^2} \left(\nu+n_2\right)
\end{split}
\end{equation}
\begin{equation} \label{part1_3}
\begin{split}
\phantom{space} \left|R_{\mu}^{3}(\xi)\right| &= \left|\dfrac{\Lambda(\xi)}{e^{\xi^2\mu^2/4}}\right|
\leq \dfrac{\nu+|\xi|^{\alpha}}{ (1-e^{-N})(e^{\xi^2\mu^2/4})} 
< 2 \,\dfrac{\nu+|\xi|^{\alpha}}{e^{\xi^2\mu^2/4}}\\ 
&=2 \left(\dfrac{\nu}{e^{\xi^2\mu^2/4}}+\dfrac{|\xi|^{\alpha}}{e^{\xi^2\mu^2/4}}\right)
< 2 \left(\dfrac{\nu}{\mu^2}+\dfrac{n_3}{ \mu^2}\right) = \frac{2}{\mu^2} \left(\nu+n_3\right)
\end{split}
\end{equation}

\setlength{\leftskip}{0pt}
\setlength{\leftskip}{0pt}

%\vspace{-0.8cm}
% 
\emph{If $N \in (0,1)$}: From the inequality \eqref{D_ineq2}, using Lemma \ref{lemma3} and Lemma \ref{lemma2}, we have,

%
%\vspace{-0.8cm}

%\setlength{\leftskip}{1cm}
%
\begin{equation} \label{part2_1}
\begin{split}
\phantom{space} \left|R_{\mu}^{1}(\xi)\right| &= \left|\dfrac{\Lambda(\xi)}{1+\xi^2\mu^2}\right|
\leq \dfrac{(\nu+|\xi|^{\alpha})N}{ (1-e^{-N})(1+\xi^2\mu^2) N} 
< 2 \,\dfrac{\nu+|\xi|^{\alpha}}{(1+\xi^2\mu^2)N}\\ 
&=\dfrac{2}{N} \left(\dfrac{\nu}{ 1+\xi^2\mu^2}+\dfrac{|\xi|^{\alpha}}{ 1+\xi^2\mu^2}\right)
< \dfrac{2}{N} \left(\dfrac{\nu}{\mu^2}+\dfrac{n_1}{ \mu^2}\right) = \frac{2}{N \, \mu^2} \left(\nu+n_1\right)
\end{split}
\end{equation}
\begin{equation} \label{part2_2}
\begin{split}
\phantom{space} \left|R_{\mu}^{2}(\xi)\right| &= \left|\dfrac{\Lambda(\xi)}{1+\xi^4\mu^2}\right|
\leq \dfrac{(\nu+|\xi|^{\alpha})N}{ (1-e^{-N})(1+\xi^4\mu^2)N} 
< 2 \,\dfrac{\nu+|\xi|^{\alpha}}{(1+\xi^4\mu^2)N}\\ 
&=\dfrac{2}{N} \left(\dfrac{\nu}{ 1+\xi^4\mu^2}+\dfrac{|\xi|^{\alpha}}{ 1+\xi^4\mu^2}\right)
< \dfrac{2}{N} \left(\dfrac{\nu}{\mu^2}+\dfrac{n_2}{ \mu^2}\right) = \frac{2}{N \, \mu^2} \left(\nu+n_2\right)
\end{split}
\end{equation}
\begin{equation} \label{part2_3}
\begin{split}
\phantom{space} \left|R_{\mu}^{3}(\xi)\right| &= \left|\dfrac{\Lambda(\xi)}{e^{\xi^2\mu^2/4}}\right|
\leq \dfrac{(\nu+|\xi|^{\alpha})N}{ (1-e^{-N})(e^{\xi^2\mu^2/4})N} 
< 2 \,\dfrac{\nu+|\xi|^{\alpha}}{N \, e^{\xi^2\mu^2/4}}\\ 
&=\dfrac{2}{N} \left(\dfrac{\nu}{e^{\xi^2\mu^2/4}}+\dfrac{|\xi|^{\alpha}}{e^{\xi^2\mu^2/4}}\right)
< \dfrac{2}{N} \left(\dfrac{\nu}{\mu^2}+\dfrac{n_3}{ \mu^2}\right) = \frac{2}{N\, \mu^2} \left(\nu+n_3\right)
\end{split}
\end{equation}

To finish the proof of the Lemma \ref{lemma4}, all that remains is to combine the expressions \eqref{part1_1}-\eqref{part2_3}.
\end{proof}

\begin{lmm} 
\label{lemma5}
Let $\xi \in \R$, $0<p<\infty$ and $0<\mu<1$, then the following inequality holds
\begin{equation}
\label{cota_gral}
\sup_{\xi \in \R} \left\vert (1+ \xi^2)^{-p/2} \left(1-\dfrac{R_{\mu}^{i}(\xi)}{\Lambda(\xi)}\right)\right\vert \le \max \left\{\mu^p,\mu^2, \mu^{p-2}\right\}, \quad i=1,2,3,
\end{equation}
where $\Lambda(\xi)$ and $R_{\mu}^{i}(\xi)$ are given in \eqref{Lambda} and \eqref{familyR} respectively.
\end{lmm}
%%%%%
\begin{proof}
Let
\begin{equation*}
\Omega_i(\xi):=(1+ \xi^2)^{-p/2} \left(1-\dfrac{R_{\mu}^{i}(\xi)}{\Lambda(\xi)}\right), \quad i=1,2,3.
\end{equation*}
Three cases are considered for the proof.

\vspace{0.5cm}

\textbf{Case 1} $ (|\xi| \geq |\xi_0|:=\frac{1}{\mu})$. Thus
%\vspace{-0.5cm}

		    \begin{equation} 
        \label{caso1}
        \Omega_i(\xi) \leq (1+ |\xi|^2)^{-p/2} \leq |\xi|^{-p} \leq |\xi_0|^{-p}= \mu^{p}, \quad i=1,2,3.
        \end{equation}

\vspace{0.5cm}		
				
 \textbf{Case 2} $(|\xi|   < 1)$. Thus 

%\vspace{-0.5cm}
		
        \begin{equation} 
        \label{caso2_1}
        \Omega_1(\xi)= \dfrac{\xi^2 \mu^2}{1+\xi^2 \mu^2} (1+ \xi^2)^{-p/2}\leq \xi^2 \mu^2 (1+ \xi^2)^{-p/2}\leq \mu^2,
        \end{equation}
				\begin{equation} 
        \label{caso2_2}
           \Omega_2(\xi)= \frac{\xi^4 \mu^2}{1+\xi^4 \mu^2} (1+ \xi^2)^{-p/2}\leq \xi^4 \mu^2 (1+ \xi^2)^{-p/2}\leq \mu^2,
        \end{equation}
				\begin{equation} 
        \label{caso2_3}
              \Omega_3(\xi)= \left(1-e^{-\frac{\xi^2 \mu^2}{4}}\right) (1+ \xi^2)^{-p/2}\leq 1-e^{\frac{-\xi^2 \mu^2}{4}}\leq \frac{\xi^2 \mu^2}{4}\leq \mu^2.
        \end{equation}
				
  \textbf{Case 3} $(1 \leq  |\xi| < |\xi_0|:=\frac{1}{\mu})$. Thus 
	
	%\vspace{-0.5cm}
				
         \begin{equation} 
         \label{caso3_1}
        \Omega_1(\xi)= \dfrac{\xi^2 \mu^2}{1+\xi^2 \mu^2} (1+ \xi^2)^{-p/2}\leq \dfrac{\left|\xi\right|^{2-p} \, \mu^2}{1+\xi^2 \mu^2}\leq \left|\xi\right|^{2-p} \, \mu^2,
         \end{equation}
				 \begin{equation} 
         \label{caso3_2}
        \Omega_2(\xi)= \dfrac{\xi^4 \mu^2}{1+\xi^4 \mu^2} (1+ \xi^2)^{-p/2}\leq \dfrac{\left|\xi\right|^{4-p} \, \mu^2}{1+\xi^4 \mu^2}\leq \left|\xi\right|^{4-p} \, \mu^2,
         \end{equation}
				\begin{equation} 
        \label{caso3_3}
              \Omega_3(\xi)= \left(1-e^{\frac{-\xi^2 \mu^2}{4}}\right) (1+ \xi^2)^{-p/2}\leq \left|\xi\right|^{-p} \dfrac{\xi^2 \mu^2}{4} \leq \left|\xi\right|^{2-p} \, \mu^2.
        \end{equation}
				 
If $0<p\leq 2,$ it results
       \begin{equation} 
         \label{subcaso1_3}
         \Omega_i(\xi)\leq \left|\xi\right|^{2-p} \, \mu^2  \leq \left|\xi_0\right|^{2-p} \, \mu^2 =\mu^p, \quad i=1,3.
         \end{equation}

If $p > 2$, it follows that 
       \begin{equation} 
         \label{subcaso2_3}
        \Omega_i(\xi)\leq \left|\xi\right|^{2-p} \, \mu^2  \leq \mu^2, \quad i=1,3.
         \end{equation}

Analogously, if $0<p\leq 4,$ it results
       \begin{equation} 
         \label{subcaso1_3bis}
         \Omega_2(\xi)\leq \left|\xi\right|^{4-p} \, \mu^2  \leq \left|\xi_0\right|^{4-p} \, \mu^2 =\mu^{p-2}.
         \end{equation}

If $p > 4$, it follows that 
       \begin{equation} 
         \label{subcaso2_3bis}
        \Omega_2(\xi)\leq \left|\xi\right|^{4-p} \, \mu^2  \leq \mu^2.
         \end{equation}

The expressions given by \eqref{caso1}-\eqref{subcaso2_3bis} are combined to obtain \eqref{cota_gral}. This concludes the proof.
\end{proof}

\begin{nt}
Notice that 
\begin{equation*} 
\lim_{\mu \to 0^+}\dfrac{R_{\mu}^{i}(\xi)}{\Lambda(\xi)}=1, \qquad i=1,2,3.
\end{equation*}
\end{nt}

\subsubsection{Analytical bound of error}

It is now possible to obtain a bound for the error estimate.

\begin{dfntn} 
The norm in the Sobolev space $H^p(\R)$, for $p>0$, is defined as: 
\begin{equation}
\label{Boundf}
\|  f\|_{H^p(\R)} := \left(\, \int\limits_{-\infty}^{\infty}{|\widehat{f}|^2 \left(1+\xi^2 \right)^{p}\, d\xi} \right)^{1/2}. 
\end{equation}
\end{dfntn}

In the literature on the stabilization of ill-posed problems, specifically in works concerning the estimation of unknown sources, different approaches to define an a priori regularization parameter selection rules can be found, which utilize a bound for the norm of the source \eqref{Boundf}. Since the source is unknown, its norm is also unknown, necessitating the assumption of a value for the bound that may not be accurate. Typically, $C=1$ is used (see, for example, \cite{Yang10b, Yang10, Yang11}), which in some way limits the range of possible sources to be determined. To address this issue, this article proposes an a priori regularization parameter selection rule based solely on the noise level in the data. The regularization parameter is normalized with respect to the maximum tolerated noise. This selection rule allows for the derivation of an analytical expression for the error bound of each estimation, as demonstrated in the following result.

\begin{thrm}[Analytical bound for the estimation error]
\label{boundestimate}
Consider the inverse problem of determining the source $f(t)$ in (\ref{transpeqn}). Let $f_{\delta ,\mu}^{i}(t)$, with $i=1,2,3$, be the regularization solutions. It is assumed that there exists $C \in \R^+$ bounding the norm of $f $ in $H^p(\R)$ for some $0<p<\infty$ \eqref{Boundf}.
Letting
\begin{equation}
\label{mureg}
    \mu^2= \left(\dfrac{\delta}{\delta_M}\right)^{2/p+2} < 1,
\end{equation} 

there exist constants $K_i \, (i=1,2,3)$ independent of $\delta$ such that,
\begin{equation*}
\|  f -f_{\delta ,\mu }^{i} \|_{L^2(\R)} \le  K_i \, \max \left\{  \left(\dfrac{\delta}{\delta_M}\right)^{2/p+2}, \left(\dfrac{\delta}{\delta_M}\right)^{p/p+2}, \left(\dfrac{\delta}{\delta_M}\right)^{p-2/p+2}\right\},  \qquad i=1,2,3.
\end{equation*}
\end{thrm}
%%%%%
\begin{proof}
\begin{equation*}
\left\| \widehat{f}(\xi )-R_{\mu}^{i} (\xi) \, \widehat{y}(\xi ) \right\|_{L^2(\R)}
= \left\| \widehat{f}(\xi) \left(1-\frac{R_{\mu}^{i} (\xi)}{\Lambda(\xi)}\right) \frac{(1+ \xi^2)^{p/2}}{(1+ \xi^2)^{p/2}} \right\|_{L^2(\R)}, 
\end{equation*}
rearranging
\begin{equation*}
\left\| \widehat{f}(\xi )-R_{\mu}^{i} (\xi) \, \widehat{y}(\xi ) \right\|_{L^2(\R)}
\le  \sup_{\xi \in \R} \left| (1+ \xi^2)^{-p/2}\left(1-\frac{R_{\mu}^{i} (\xi)}{\Lambda(\xi)}\right)    \right| 		\left\| \widehat{f}(\xi) (1+ \xi^2)^{p/2} \right\|_{L^2(\R)}.
\end{equation*}

By using the definition of the norm in the Sobolev space $H^p(\R)$ given by the expression \eqref{Boundf}, it results, 
\begin{equation}
\label{ineq2teo2}
\left\|\widehat{f}(\xi )-R_{\mu}^{i} (\xi) \, \widehat{y}(\xi ) \right\|_{L^2(\R)}
\le \sup_{\xi \in \R} \left| (1+ \xi^2)^{-p/2}\left(1-\frac{R_{\mu}^{i} (\xi)}{\Lambda(\xi)}\right)  \right|  \left\|  f  \right\|_{H^p(\R)}.
\end{equation}

By the triangle inequality, we have
\begin{equation}
\label{dnormas}
 \left\| \widehat{f} - \widehat{f_{\delta ,\mu }^{i}} \right\|_{L^2(\R)}   \leq \left\|  \widehat{f} - R_{\mu}^{i} (\xi) \, \widehat{y}(\xi ) \right\|_{L^2(\R)}   +  \left\| R_{\mu}^{i} (\xi) \, \widehat{y}(\xi )  - \widehat{f_{\delta ,\mu }^{i}}\right\|_{L^2(\R)}.  
\end{equation}

From the inequalities (\ref{ineq2teo2})-(\ref{dnormas}), together with the definition of the regularized source \eqref{ffregtransf} in frequency space, we obtain,
\begin{equation}
\label{eq_aux_dem1}
\left\|  \widehat{f}  -\widehat{f_{\delta ,\mu }^{i}} \right\|_{L^2(\R)}
\le  \sup_{\xi \in \R} \left| (1+ \xi^2)^{-p/2}\left(1-\frac{R_{\mu}^{i}(\xi)}{\Lambda(\xi)}\right)    \right|  \left\|  f \right\|_{H^p(\R)}+\sup_{\xi \in \R} \left|R_{\mu}^{i}(\xi)\right|\left\|\widehat{y}-\widehat{y}_{\delta}\right\|_{L^2(\R)},
\end{equation}

Parseval's identity \cite{Marks09} is used in \eqref{eq_aux_dem1}, along with the fact that $C$ bounds $f$ in the norm $H^p(\R)$ and the assumption that the error in the data is bounded $(\left\|y-y_{\delta }\right\|_{L^2(\R)}=\left\|\widehat{y}-\widehat{y }_{\delta }\right\|_{L^2(\R)} \leq \delta)$. Then,
\begin{equation}
\label{eq_aux_dem2}
\left\| f -f_{\delta ,\mu }^{i} \right\|_{L^2(\R)} =\left\|  \widehat{f} -\widehat{f_{\delta ,\mu }^{i}}  \right\|_{L^2(\R)}
\le C \, \sup_{\xi \in \R} \left| (1+ \xi^2)^{-p/2}\left(1-\frac{R_{\mu}^{i}(\xi)}{\Lambda(\xi)}\right)    \right|  + \delta \, \sup_{\xi \in \R} \left|R_{\mu}^{i}(\xi)\right|.
\end{equation}  

By virtue of Lemmas \ref{lemma4}, \ref{lemma5}, the expression \eqref{eq_aux_dem2} can be rewritten as,
\begin{equation}
\label{eq_aux_dem3}
\left\| f -f_{\delta ,\mu }^{i}  \right\|_{L^2(\R)} 
< C \, \max\left\{\mu^2, \mu^p, \mu^{p-2} \right\}  + \frac{\delta}{\mu^2} \,M_i,
\end{equation}  
where $M_i$  is given by \eqref{M's} for $i=1,2,3$. Using that $\mu^2= \left(\dfrac{\delta}{\delta_M} \right)^{2/p+2}$ \eqref{mureg} the inequality \eqref{eq_aux_dem3} leads to
\begin{equation*}
\left\| f -f_{\delta ,\mu }^{i}  \right\|_{L^2(\R)}
< C \, \max\left\{\left(\dfrac{\delta}{\delta_M}\right)^{2/p+2}, \left(\dfrac{\delta}{\delta_M}\right)^{p/p+2}, \left(\dfrac{\delta}{\delta_M}\right)^{p-2/p+2}\right\}  + \delta_M  \,M_i \, \left(\dfrac{\delta}{\delta_M}\right)^{p/p+2},
\end{equation*} 
equivalently
\begin{equation}
\label{cota_final}
\left\|  f -f_{\delta ,\mu }^{i} \right\|_{L^2(\R)} <  K_i \, \max\left\{\left(\dfrac{\delta}{\delta_M}\right)^{2/p+2}, \left(\dfrac{\delta}{\delta_M}\right)^{p/p+2}, \left(\dfrac{\delta}{\delta_M}\right)^{p-2/p+2}\right\},
\end{equation}
where $K_i=C+\delta_M \,M_i$, for $i=1,2,3$. Therefore, the proof is complete.
\end{proof}

\begin{nt}
Note that the bound \eqref{cota_final} satisfies
\begin{equation*}
\| f -f_{\delta, \mu}^{i} \|_{L^2(\R)} \longrightarrow 0 \,\,\,\,\,\, \text{if} \,\,\,\,\,\, \delta \longrightarrow 0,
  \end{equation*}
which means that the estimate $f_{\delta, \mu}^{i}$ converges to the source function $f$ as the noise in the data $\delta$ tends to $0$.
\end{nt}
\begin{nt}
Note that the case $p=\infty$ is not included in the hypotheses of the Theorem \ref{boundestimate}. The reason is that, in this case,
\begin{equation*}
  \| f -f_{\delta, \mu}^{i} \|_{L^2(\R)} =\left\| \widehat{f} -\widehat{f_{\delta ,\mu }^{i}} \right\|_{L^2(\R)} < K_i < C \nrightarrow 0.
  \end{equation*}
\end{nt}
%
%\begin{nt}
%
%Note that the bound obtained for the regularization error \eqref{cota_final} is of H\"older type and depends only on the smoothness of the source and the parameters of the mathematical model.
%
%\end{nt}
%
%
%\begin{rmrk}
%A case of special interest in the literature is for $p=2$. In this case, the expression \eqref{cota_final} reduces to:
%
%\begin{equation*}
%\left\|  f -f_{\delta ,\mu }^{i} \right\|_{L^2(\R)} <  K_i \, \sqrt{\dfrac{\delta}{\delta_M}}.
%\end{equation*}
%
%\end{rmrk}

\section{Numerical examples and discussion of results}

In this section, we present numerical examples with different characteristics and discuss the results obtained.

\subsection{Numerical examples}

We present specific examples for the estimation of the source $f$ in $\R$.
For each example discussed, we select different values for the parameters $(\omega, \beta, \nu, x_0)$ of the source identification problem. In addition, to simulate noise in the data, a range of values for standard deviation $\epsilon$ is utilized. Time is discretized into a uniform mesh, and a data set ${y_{\eta_1},...,y_{\eta_N}}$ is generated from the evaluation at $x_0$ of the solution $u(x,t)$. Afterwards, noise is added, specifically:
\begin{equation*}
\label{noisydata}
y_{\eta_i} = y(t_i) + \eta_i , \quad i=1,...,N, \,\, t_i \in {\cal G},
\end{equation*}
where ${\cal G}$ represents a uniform discretization of $\R$, and $\eta_i$, for $i=1,...,N$ are independent realizations of the normally distributed random variable ${\eta}$ with mean 0, standard deviation $\epsilon$ and noise level $\delta=\delta(\epsilon)$.

%The error $\left\|y(t) -y_{\delta}(t)\right\|_{L^2(\R)}$ is numerically calculated using Simpson's integration method. This calculation dependent del ruido y de la discrtizacion ${\cal G}$ considerada 
%

Afterwards, ${\widehat{y}_{\eta_1},\ldots,\widehat{y}_{\eta_N}}$ is computed using the FFT (Fast Fourier Transform) \cite{Van92}, and the regularization solutions $f^{i}_{\delta ,\mu}$ for $i=1,2,3$ given in \eqref{ffreg2} are obtained using the inverse Fast Fourier Transform \cite{Van92}, where the regularization parameter $\mu$ is chosen according to \eqref{mureg}, that is,
\begin{equation*}
\mu^2= \ds \left(\frac{\delta}{\delta_M}\right)^{2/p+2}.
\end{equation*}

In practice, the maximum tolerance for the error in the data, $\delta_M$, given in \eqref{noiselevel}, is determined based on calibration, instrumentation errors, and measurement inaccuracies associated with the instruments used. To ensure that the regularization parameter does not exceed 1, it is defined as $\delta$ plus one unit. In other words,

\begin{equation*}
\delta_M= 1+ \delta.
\end{equation*}

The results of the estimated sources without regularization and those obtained after applying the regularization methods presented in this article are plotted, for each example. A table is also included detailing the relative errors incurred in the estimations by each of the regularization operators used to address the ill-posed problem. The table shows the errors for a set  of five data noise levels, each represented by the standard deviations ${\epsilon_1;\ldots;\epsilon_5}= {10^{-1}, 10^{-2}, 10^{-3}; 10^{-4}; 10^{-5}}$, chosen based on the range of values of the solution.

\begin{rmrk}
In some examples, the $p$ consideref for the recovery does not correspond to the space $H^p(\R)$ to which the source belongs. These examples were included to show that the recovery is reasonable even in those cases.
\end{rmrk}
\begin{rmrk}
In the numerical experiments, simulations were conducted for several values of $x_0$ (across different positions). However, no significant differences were observed in the results for varying $x_0$. Therefore, to simplify the presentation and focus on the most relevant aspects of the study, only the results for a specific value of $x_0$ are shown. 
\end{rmrk}

We analyze two examples of one-dimensional source estimation with different characteristics: one discontinuous and the other continuous.
In particular, the example with the discontinuous source does not belong to $H^p(\R)$ for $p>0$ and therefore the results of Theorem \ref{boundestimate} do not apply. Nevertheless, this example is included since, even in this case, the approximation improves when the regularization operators are used.

\begin{xmpl}
\label{example1}

\textbf{\textit{[Discontinuous source]}} For this example, the following parameters are considered: $\omega=0.1$; $\beta= 0.9$; $\nu=1$;  $\alpha=0.9$; $N=256$ and $x_0=0.5$.
The standard errors to generate the data noise $\epsilon \in \left\{0.05 \, , 0.1\, , 0.15\, , 0.2 \right\}$
and the source to estimate is:

\begin{equation}
\label{fuente_ex1}
f(t)=\begin{cases}
-1, \qquad & 0 \leq t<2.5, \\
  1, \qquad & 2.5 \leq t <5, \\
  -1, \qquad & 5 \leq t <7.5, \\
1, \qquad & 7.5 \leq t \leq 10, \\
  0, \qquad & \text{in another case}.
\end{cases}
\end{equation}
\end{xmpl}

The source described in \eqref{fuente_ex1}, retrieved in Example \ref{example1} is significant in the context of signal theory. This type of function is commonly encountered in source retrieval problems, as illustrated in \cite{Yang11,Yang14}. The relevance arises primarily from its discontinuous nature, which can lead to the Gibbs phenomenon \cite{Elden00,Van92}. This phenomenon indicates that when a Fourier series is developed for a function that is not continuous in the considered region, there may be poor precision in the neighborhoods of the discontinuities.

%\vspace{-0.3cm}
\begin{figure}[h!]
\begin{center}
\includegraphics[width=0.46\textwidth]{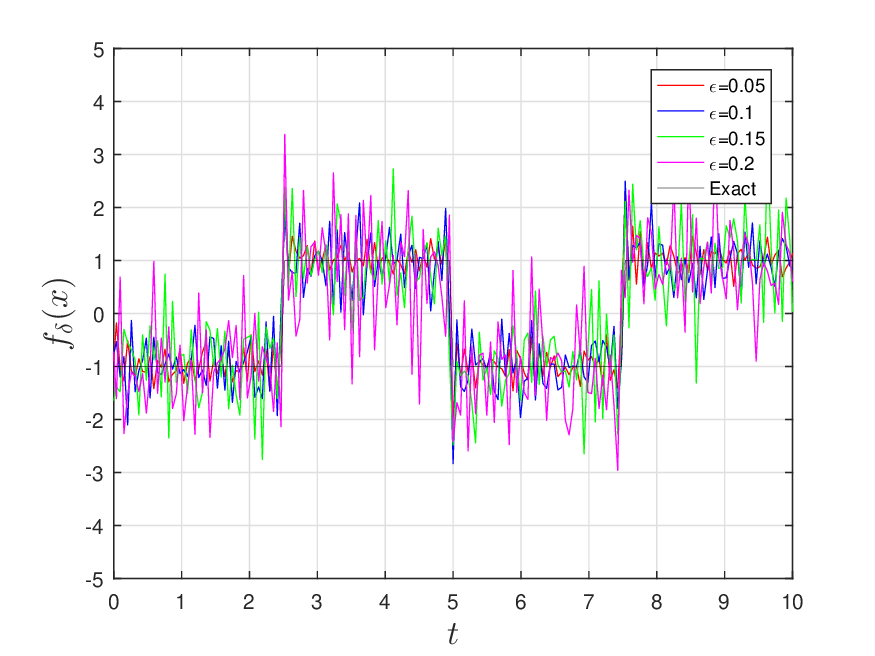}
\includegraphics[width=0.46\textwidth]{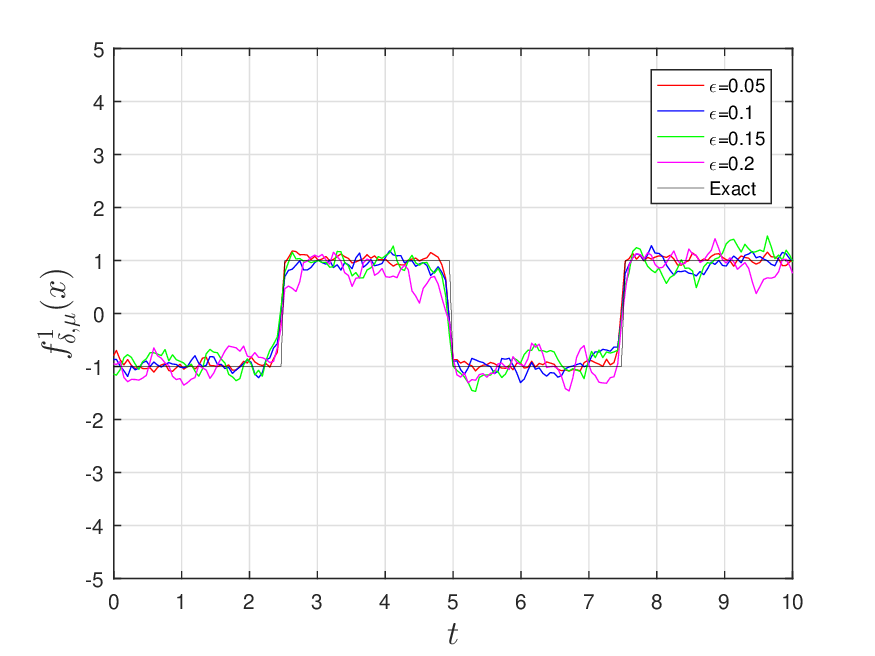}
\includegraphics[width=0.46\textwidth]{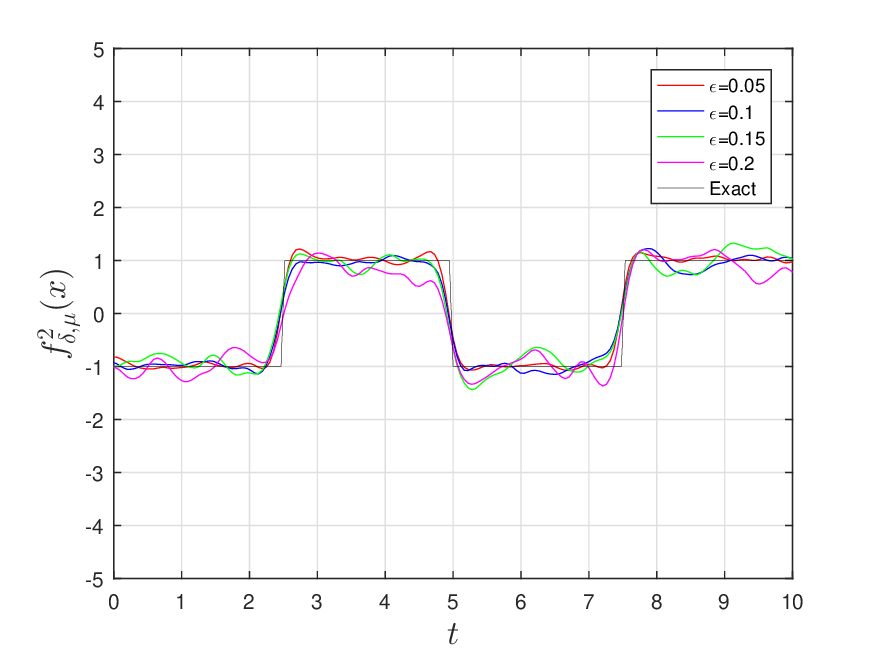}
\includegraphics[width=0.46\textwidth]{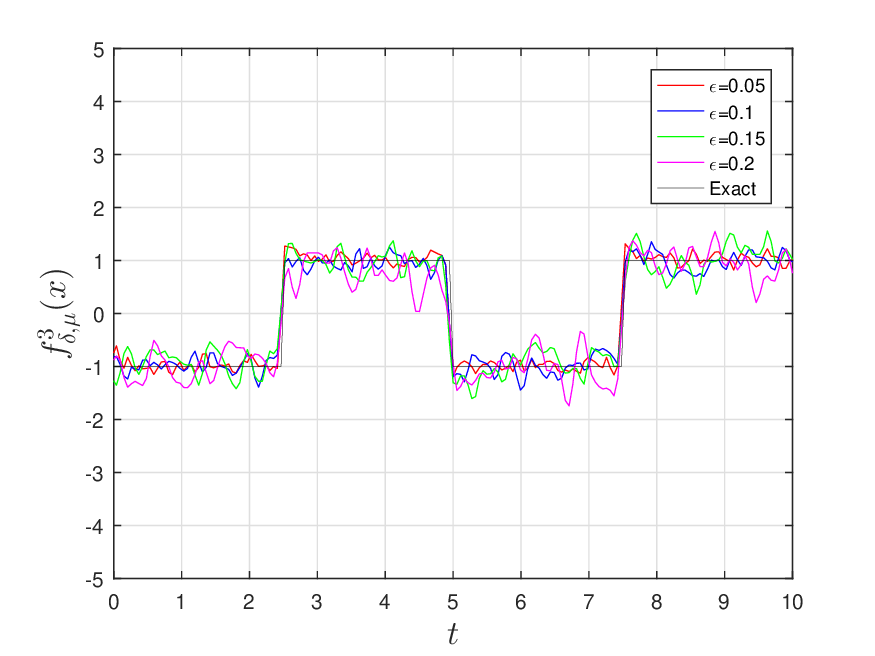}
\vspace{-0.2cm}
\caption{Example \ref{example1}: Non-regularized source with $x_0=0.5$ (top-left); Regularized sources with $x_0=0.5$ and $p=1$, using $R^{1}_\mu$ (top-right), $R^{2}_\mu$ (bottom-left), $R ^{3}_\mu$ (bottom-right) for different noise levels.}
\vspace{-0.3cm}
\label{CurvaCuadrada}
\end{center}
\end{figure}

Table \ref{tableej1} lists the relative errors for the source estimation obtained with the regularization methods used. For example, in the case of $\epsilon=10^{-5}$, the relative error for the regularized source estimation using $R^{1}_{\mu}$ is $2.11 \%$, while for $R^{2}_{\mu}$ it is $2.97 \%$, and for $R^{3}_{\mu}$ it is $3.15\%$.
%\vspace{-0.10cm}
\begin{table}[h!]
\begin{center}
{\begin{tabular}{c}\toprule
Relative errors\\
{\begin{tabular}{lccc} \toprule
\,\, $\epsilon$ &  $ \| f-f^{1}_{\delta,\mu}\|/ \left\|f\right\|$   & $ \| f-f^{2}_{\delta,\mu}\|/ \left\|f\right\|$  & $ \| f-f^{3}_{\delta,\mu}\|/ \left\|f\right\|$   \\ \midrule
$10^{-1}$ & 0.2275  & 0.2307  & 0.2591 \\
$10^{-2}$ & 0.1341  & 0.1459   & 0.1587 \\
$10^{-3}$  & 0.0849 &	0.0996  & 0.1033\\
$10^{-4}$  & 0.0478  & 0.0513  & 0.0591   \\
$10^{-5}$  & 0.0211  & 0.0297 & 0.0315 \\\bottomrule
\end{tabular}}
\end{tabular}}
\end{center}
\vspace{-0.30cm}
\caption{Example \ref{example1}: Relative estimation errors for $ x_0 = 1 $ and $ p = 1.$}
\vspace{-0.30cm}
\label{tableej1}
\end{table}

%%%%%%%%%%%%%%%%%%%%%%%%%%%%%%%%%%%%%%%%%%%%%%%%%%%%%%%%%%%%%%%%%%%%%

\begin{xmpl}
\label{example2}

\textbf{\textit{[Continuous source]}} The following parameters are considered: $\omega=0.01$; $\beta= 0.5$; $\nu=1.51$;  $\alpha=0.3$; $N=256$ and $x_0=10$.
The standard errors to generate the data noise are $\epsilon \in \left\{0.05 \, , 0.1\, , 0.15\, , 0.2 \right\}$
and the source to estimate in this case is:

\begin{equation}
\label{fuente_ex2}
f(t)=\begin{cases}
6.51 e^{-t}, \qquad & 0 \leq t\leq 10, \\
  0, \qquad & \text{in another case}.
\end{cases}
\end{equation}
\end{xmpl}

%\vspace{-0.3cm}
\begin{figure}[h!]
\begin{center}
\includegraphics[width=0.46\textwidth]{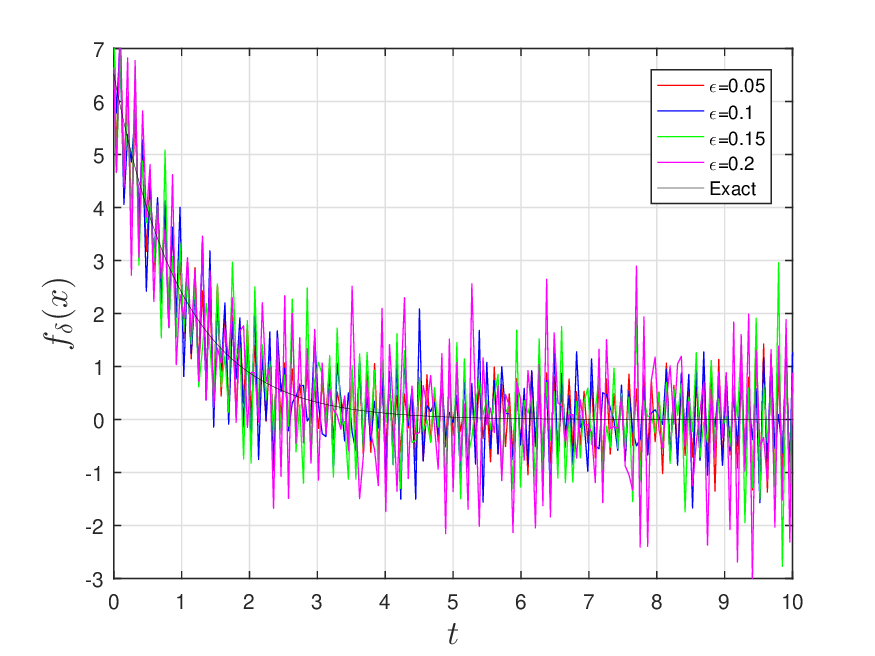}
\includegraphics[width=0.46\textwidth]{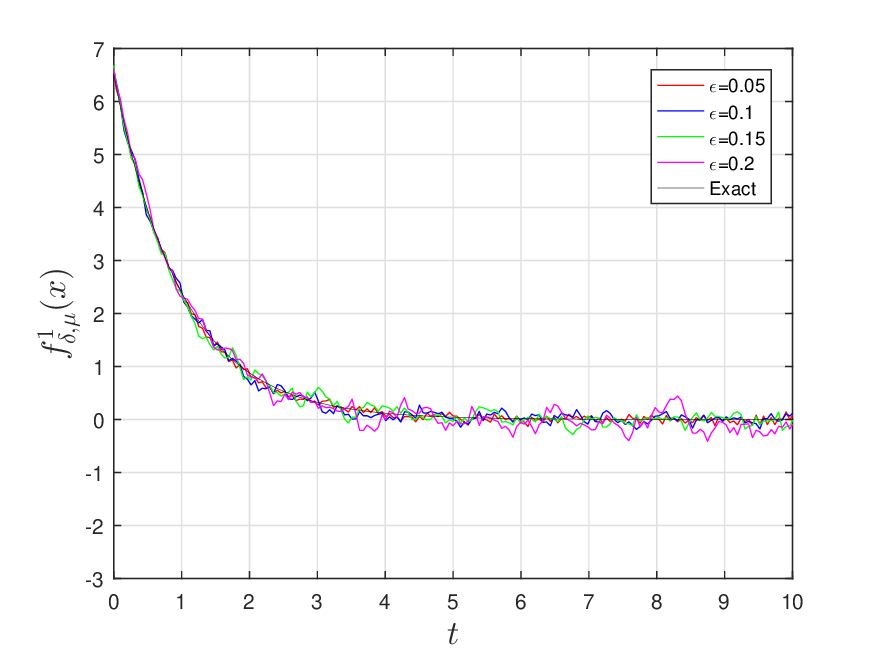}
\includegraphics[width=0.46\textwidth]{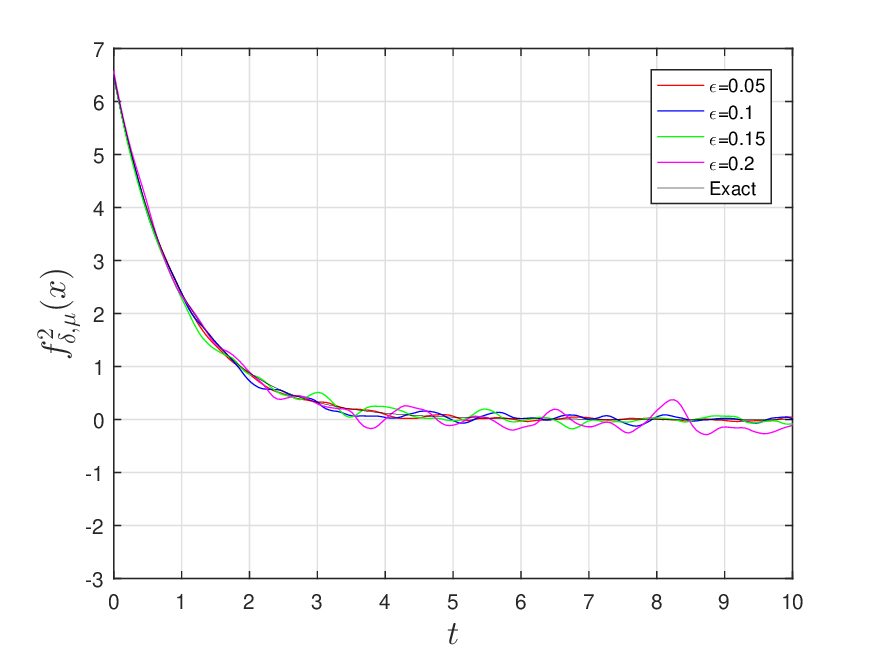}
\includegraphics[width=0.46\textwidth]{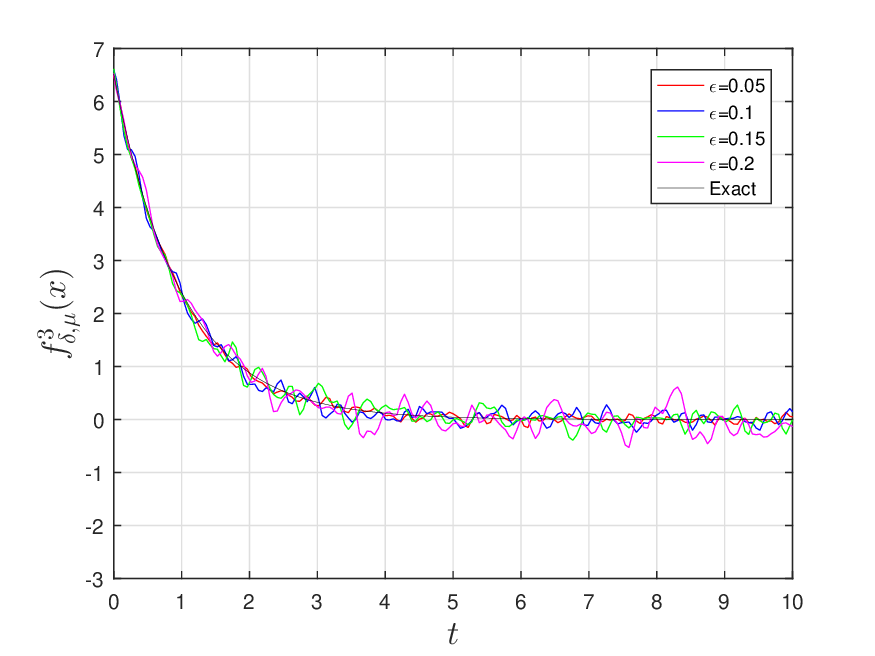}
\vspace{-0.2cm}
\caption{Example \ref{example1}: Non-regularized source with $x_0=10$ (top-left); Regularized sources with $x_0=10$ and $p=2$, using $R^{1}_\mu$ (top-right), $R^{2}_\mu$ (bottom-left), $R ^{3}_\mu$ (bottom-right) for different noise levels.}
\vspace{-0.3cm}
\label{CurvaExponencial}
\end{center}
\end{figure}

Table \ref{tableej2} lists the relative errors for the source estimation obtained with the regularization methods used. For example, in the case of $\epsilon=10^{-5}$, the relative error for the regularized source estimation using $R^{1}_{\mu}$ is $2.43 \%$, while for $R^{2}_{\mu}$ it is $2.64 \%$, and for $R^{3}_{\mu}$ it is $2.96\%$.

%\vspace{-0.10cm}
\begin{table}[h!]
\begin{center}
{\begin{tabular}{c}\toprule
Relative errors\\
{\begin{tabular}{lccc} \toprule
\,\, $\epsilon$ &  $ \| f-f^{1}_{\delta,\mu}\|/ \left\|f\right\|$   & $ \| f-f^{2}_{\delta,\mu}\|/ \left\|f\right\|$  & $ \| f-f^{3}_{\delta,\mu}\|/ \left\|f\right\|$   \\ \midrule
$10^{-1}$ & 0.3200  & 0.3230  & 0.3310 \\
$10^{-2}$ & 0.1637  & 0.1689   & 0.1703 \\
$10^{-3}$  & 0.0932 &	0.0975  & 0.099\\
$10^{-4}$  & 0.0499  & 0.0521  & 0.0573   \\
$10^{-5}$  & 0.0243  & 0.0264 & 0.0296 \\\bottomrule
\end{tabular}}
\end{tabular}}
\end{center}
\vspace{-0.30cm}
\caption{Example \ref{example2}: Relative estimation errors for $ x_0 = 1 $ and $ p = 1.$}
\vspace{-0.30cm}
\label{tableej2}
\end{table}

\subsection{Discussion of results}

The numerical examples demonstrate that the three regularization operators presented in this paper effectively address the instability issues in the source estimation solution. Each of the three
operators studied smooths the inverse operator, resulting in a stable recovery. 
It is noteworthy that this stability is maintained regardless of the general characteristics of the function being determined, as demonstrated by the satisfactory results achieved even for the discontinuous source. Additionally, the three regularization strategies show to be robust, as the error committed in the estimation decreases as the error in the data decreases.

While no significant differences are observed between the performance of these, the operator $R^{1}_{\mu}$ yielded the best results, allowing for the estimation of the source with the lowest relative error at the different disturbance levels considered.

\section{Conclusions}

This paper addresses a specific case of the inverse source identification problem, focusing on estimating the time-dependent source term in a time-fractional linear parabolic equation using noisy time-dependent data. An analytical solution to the estimation problem is provided, demonstrating that the problem is ill-posed due to the instability of the solution.

To counteract this instability, three families of regularization operators are defined, specifically designed to counterbalance the inestability factor of the inverse operator. Additionally, a rule for selecting the regularization parameters is proposed, based on the noise level in the data and the degree of smoothness of the source to be identified. It is shown that this parameter selection rule leads to stable methods. A bound for the estimation errors is derived, which is optimal.

Two numerical examples of source recovery from different Hilbert spaces are presented. In both cases, the regularization approaches exhibit good performance. A comparison is also carried out between the sources reconstructed using the inverse operator and those obtained through the proposed regularization operators, showing that the latter provide accurate estimations of the considered sources. Moreover, the results suggest that the smoother the high-frequency filtering, the better the performance. This is an interesting aspect that deserves further analysis.

%% The Appendices part is started with the command \appendix;
%% appendix sections are then done as normal sections

\appendix
\section{Additional Results}
\label{apendice}

\begin{lmm}
\label{lemma1}
Let $z \in \C$ with $\Re(z)>0$ then
$\ds \left\vert \dfrac{1}{1-e^{-z }} \right\vert \le \dfrac{1}{1-e^{-\Re(z)}}$. In addition $\Re(\sqrt{z})\geq\sqrt{\Re(z)}.$
\end{lmm}
\begin{proof}

Euler's formula using for complex numbers and we find that,
\begin{equation*}
\ds{\left|1-e^{-z}\right|^2= \left(1-e^{-\Re(z)}\right)^2+2e^{-\Re(z)}\left(1-cos\left(\Im(\omega)\right)\right)} \geq \left(1-e^{-\Re(\omega)}\right)^2,
\end{equation*}
therefore, 
\begin{equation*}
\left|1-e^{-z}\right| \geq 1-e^{-\Re(z)} \Longrightarrow \left\vert\dfrac{1}{1-e^{-z}} \right\vert \le \dfrac{1}{1- e^{-\Re(\omega)}}.
\end{equation*}

The proof of the second inequality follows directly from properties of complex numbers. It suffices to note that,
\begin{equation*}
\Re(\sqrt{z})=\sqrt{\dfrac{\Re(z)+\left|z\right|}{2}}\geq\sqrt{\Re(z)}.
\end{equation*}
\end{proof}

\begin{lmm}
\label{lemma2}
The function $g:\R^{+} \to \R$ defined by
$\ds g(x):=
\begin{cases}
   \dfrac{x}{1-e^{-x}}, & 0<x<1, \vspace{0.15cm} \\
   \dfrac{1}{1-e^{-x}}, & x\geq1,
   \end{cases}$
satisfies $g(x) < 2 $.
\end{lmm}
\begin{proof}
First, consider $g$ at $(0,1)$. Differentiating, in this case, we have
\begin{equation*}
g'(x)= \left(\dfrac{x}{1-e^{-x}}\right)'= \dfrac{1-e^{-x}(1+x)}{(1 -e^{-x})^2} > 0,
\end{equation*}

the function $g$ is increasing in $(0,1)$. Then $\dfrac{x}{1-e^{-x}} \leq \dfrac{1}{1-e^{-1}}.$

On the other hand, for $x\geq1$, we have
\begin{equation*}
g'(x)=\left (\dfrac{1}{1-e^{-x}}\right) '=\dfrac{-e^{-x}}{(1-e^{-x} )^2} < 0,
\end{equation*}

the function $g$ is decreasing $ \forall x\geq1 $. Then $\dfrac{1}{1-e^{-x}} \leq \dfrac{1}{1-e^{-1}}.$

So, $g(x)\leq \dfrac{1}{1-e^{-1}} < 2$,
\end{proof}

%% \section{}
%% \label{}

%% If you have bibdatabase file and want bibtex to generate the
%% bibitems, please use
%%
%%  \bibliographystyle{elsarticle-num} 
%%  \bibliography{<your bibdatabase>}

%% else use the following coding to input the bibitems directly in the
%% TeX file.

\end{document}